\documentclass[11pt,reqno]{amsart}
\usepackage{amsaddr}
\usepackage[margin=1in]{geometry}

\usepackage{amssymb}
\usepackage{amsthm}
\usepackage{mathtools}
\usepackage{framed}
\usepackage{latexsym}
\usepackage{pb-diagram}
\usepackage{mathrsfs}
\usepackage{graphicx}
\usepackage[all]{xy} 
\usepackage{color}

\usepackage{subcaption}
\usepackage[colorlinks=true,urlcolor=blue,linkcolor=blue,citecolor=blue]{hyperref}

\usepackage{tikz}
\usetikzlibrary{backgrounds,calc,positioning,shapes, shadows,arrows,fit, automata}
\usepackage{tikz-cd}
\usetikzlibrary{arrows}
\tikzset{
    myarrow/.style={->, >=latex', shorten >=1pt, thick,orange}
}

\newcommand{\N}{\mathbb{N}}
\newcommand{\Z}{\mathbb{Z}}

\newcommand{\R}{\mathbb{R}}

\newcommand{\K}{\mathbb{K}}

\newcommand{\mc}{\mathcal}
\newcommand{\mb}{\mathbb}
\newcommand{\mf}{\mathfrak}

\renewcommand{\b}{\beta}

\renewcommand{\d}{\delta}
\newcommand{\e}{\varepsilon}

\newcommand{\s}{\sigma}
\newcommand{\ph}{\varphi}
\renewcommand{\t}{\tau}

\newcommand{\Om}{\Omega}

\newcommand{\La}{\Lambda}

\DeclareMathOperator{\id}{id}
\newcommand{\im}{\operatorname{im}}

\newcommand{\set}[1]{\left\{#1\right\}}
\renewcommand{\r}{\rightarrow}
\newcommand{\xr}{\xrightarrow}
\newcommand{\xhr}{\xhookrightarrow}
\newcommand{\hr}{\hookrightarrow}
\newcommand{\p}{\partial}
\newcommand\aeq{\stackrel{\smash{\scriptscriptstyle\rightarrow}}{=}}

\newcommand{\norm}[1]{\|#1\|} 
\renewcommand{\emptyset}{\varnothing}

\newcommand{\pow}{\operatorname{Pow}}

\newcommand{\G}{\mathfrak{G}}		%digraphs
\newcommand{\V}{\mathcal{V}}		 

\newcommand{\Ncal}{\mathcal{N}}		%networks
\newcommand{\Rsc}{\mathscr{R}}		%correspondences

\newcommand{\dn}{d_{\mathcal{N}}}	
\newcommand{\db}{d_{\operatorname{B}}} %Bottleneck
\newcommand{\di}{d_{\operatorname{I}}} %Interleaving
\newcommand{\dis}{\operatorname{dis}}
\newcommand{\nr}{\operatorname{nr}}
\newcommand{\dgm}{\operatorname{Dgm}}
\newcommand{\rank}{\operatorname{rank}}
\newcommand{\card}{\operatorname{card}}
\newcommand{\pers}{{\operatorname{\mathbf{Pers}}}}

\newcommand{\U}{\mathcal{U}}

\newtheorem{theorem}{Theorem}

\newtheorem{proposition}[theorem]{Proposition}

\theoremstyle{definition}
\newtheorem{definition}{Definition}
\newtheorem{example}[theorem]{Example}
\newtheorem{remark}[theorem]{Remark}
\newtheorem{conjecture}{Conjecture}
\newtheorem{claim}{Claim}

\newcommand{\ceil}[1]{\left \lceil{#1}\right \rceil }

\newenvironment{subproof}[1][\proofname]{%
  \begin{proof}[#1]%
}{%
  \end{proof}%
}

\title{Persistent Path Homology of Directed Networks}
\author{Samir Chowdhury and Facundo M\'{e}moli}

\begin{document}

\date{\today}
\email{chowdhury.57@osu.edu,memoli@math.osu.edu}

\address{Department of Mathematics,
The Ohio State University. 
100 Math Tower,
231 West 18th Avenue, 
Columbus, OH 43210. 
Phone: (614) 292-4975,
Fax: (614) 292-1479 }

\begin{abstract}
While standard persistent homology has been successful in extracting information from metric datasets, its applicability to more general data, e.g. directed networks, is hindered by its natural insensitivity to asymmetry. We study a construction of homology of digraphs due to Grigor'yan, Lin, Muranov and Yau, and extend this construction to the persistent framework. The result, which we call persistent path homology, can provide information about the digraph structure of a directed network at varying resolutions. 
Moreover, this method encodes a rich level of detail about the asymmetric structure of the input directed network. We test our method on both simulated and real-world directed networks and conjecture some of its characteristics.
\end{abstract}

\maketitle

\tableofcontents

\section{Introduction}
\label{sec:intro}
In recent years, the advent of sophisticated data mining tools has led to rapid growth of network datasets in the sciences. The recently completed Human Connectome Project (2010-2015, \url{http://www.humanconnectome.org/}), aimed at mapping the network structure of the human brain, is just one example of a large-scale network data acquisition project. The availability of such network data coincides with a time of steady growth of the mathematical theory of \emph{persistent homology}, which aims to study the ``shape" of data and thus appears to be a good candidate for analysing network structure. This connection is being developed rapidly \cite{sizemore2015classification, sizemore2016closures, petri2014homological, petri2013topological, giusti2015clique, giusti2016two, dlotko2016topological, masulli2016topology, dowker-arxiv}, but this exploration is far from complete. 
From a theoretical perspective, there still exists a gap in our understanding of how to feed network data, which may be asymmetric, into the method of persistent homology, which, in conventional understanding, takes symmetric data as input. The ``easy" solution of forcing a symmetrization on the data may cause its network identity to be lost, and thus a more satisfactory solution would have to be sensitive to asymmetry. This question has received some attention in \cite{dlotko2016topological, masulli2016topology, dowker-arxiv}, and in this paper, we complement the existing literature by presenting a pipeline for associating a persistent homology signature to network data that is sensitive to asymmetry.

%%%%% figure %%%%
\begin{figure}
\centering
\includegraphics[scale=0.5]{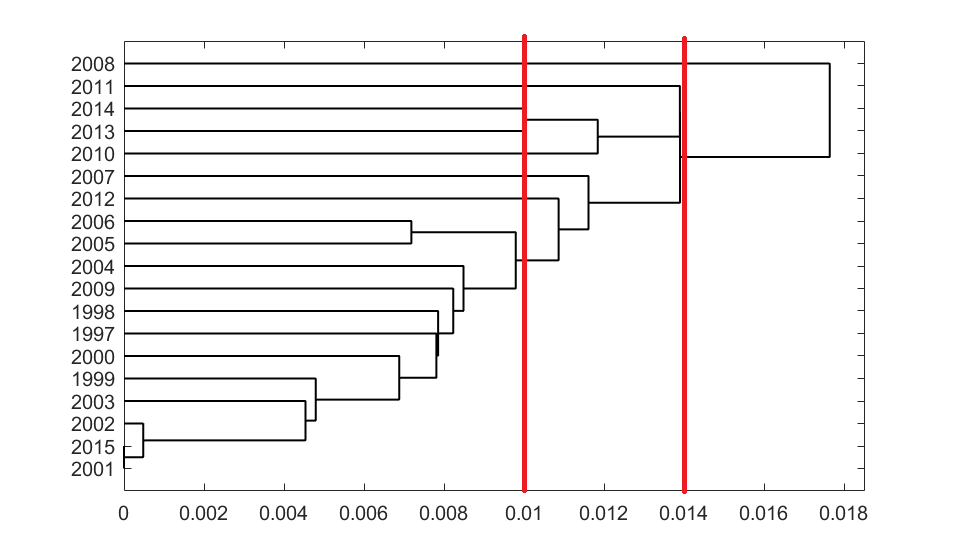}
\caption{Single linkage dendrogram obtained from bottleneck distances between 1-dimensional PPDs of the networks of interaction of the U.S. economic sectors data for 19 consecutive years. The vertical slice at resolution 0.014 shows the separation of 2008 (the height of the financial crisis) from other years. The slice at resolution 0.01 shows how the years leading up to the financial crisis are clustered together. Details are provided in \S\ref{sec:us-econ}.}
\label{fig:econ}
\end{figure}
%%%%%%%%%%%%%%

Closely related to the problem of extending persistent homology constructions to the setting of asymmetric networks is the problem of defining \emph{hierarchical clustering methods} that accept network data and are sensitive to the possible asymmetry in them. It has been observed \cite{carlsson2013axiomatic,clust-net} that in contrast to the symmetric setting \cite{clust-um,carlsson2013classifying} (exemplified by finite metric spaces) in which essentially a unique method satisfies certain axioms, in the asymmetric setting a whole continuum of functorial methods satisfies similar axioms. 

\subsubsection*{The pipeline in the case of symmetric data.}
The standard pipeline \cite{dgh-pers} for computing persistent homology of a dataset is as follows:
\begin{itemize}
\item Given a finite metric space representation of a dataset, compute a nested sequence of simplicial complexes (called a \emph{filtered simplicial complex} or a \emph{filtration}). 
\item Apply the simplicial homology functor with field coefficients (typically $\mathbb{Z}_2$) to obtain a sequence of vector spaces with linear maps, called a \emph{persistent vector space}.
\item Compute the \emph{persistence diagram} (also called a \emph{barcode}) of this persistent vector space.
\end{itemize}

Recall that a \emph{finite metric space} is defined to be a finite set $X$ and a metric function $d_X:X\times X \r \R$ satisfying the following for any $x,y,z\in X$: (1) $d_X(x,y)\geq 0$, (2) $d_X(x,y)=0 \iff x=y$, (3) $\boldsymbol{d_X(x,y)=d_X(y,x)}$ (symmetry), and (4) $d_X(x,z)\leq d_X(x,y)+d_X(y,z)$ (triangle inequality). Also recall that given a finite set $X$, a \emph{simplicial complex} on $X$ is a subcollection of the nonempty subsets of $X$ such that whenever $\s=[x_0,x_1,\ldots,x_p]$ belongs to this collection, any $\t\subseteq \s$ belongs to the collection as well. For each $p\in \Z_+$, any element of a simplicial complex consisting of $p+1$ vertices of the underlying set is called a \emph{$p$-simplex}. For $p>0$, different orderings of the vertex set of a $p$-simplex are said to be equivalent if they differ by an even permutation. The two resulting equivalence classes are then referred to as \emph{orientations} of that particular simplex \cite[pg. 26]{munkres-book}. 
We also recall that the most common way to produce a filtration from a metric space is the \emph{Rips filtration} \cite{hausmann1995vietoris}: given a metric space $(X,d_X)$ and any $\d \geq 0$, one defines $\mf{R}^\d_X:=\set{\s \subseteq X : \s \neq \emptyset, \max_{x,x'\in \s}d_X(x,x') \leq \d}.$
We do not recall terms related to persistence here, and instead refer the reader to \S\ref{sec:back-ph} for details.

\subsubsection*{The case of asymmetric data.}
Consider now the problem of letting the input dataset be a (directed/asymmetric) network. For the purposes of this paper, a \emph{network} is a finite set $X$ together with a \emph{weight function} $A_X:X\times X \r \R_+$ such that for any $x,y\in X$: (1) $A_X(x,y)\geq 0$, and (2) $A_X(x,y)=0 \iff x=y$. We will refer to such an object as a \emph{dissimilarity network} \cite{carlsson2013axiomatic,clust-net,nets-allerton,nets-icassp}, and note that such objects have been of interest since at least \cite{hubert1973min}. Neither of the two preceding conditions is essential to our constructions, but we still impose them for simplicity. The key element that we wish to address is the lack of a symmetry assumption. Some ways of getting around this obstruction are as follows:
\begin{itemize}
\item Assume that the dataset is symmetric, and apply the Rips filtration \cite{petri2013topological, giusti2015clique}.
\item Use the Rips filtration on an asymmetric network, but notice that given a network $(X,A_X)$ and two points $x,x'\in X$, the simplex $[x,x']$ appears at $\max(A_X(x,x'),A_X(x',x))$ \cite{dowker-asilo}. Stated differently, the Rips filtration forces a \emph{max-symmetrization} step on the input data.
\item Use an alternative construction called the Dowker filtration, which produces a filtered simplicial complex in a manner that is sensitive to asymmetry \cite{dowker-arxiv}.
\end{itemize}

In each case, the output is a filtered simplicial complex, to which one then applies the simplicial homology functor and obtains a persistent vector space. The entire sequence of events is referred to as the \emph{Rips/Dowker persistent homology method}, depending on the choice of filtration. Unfortunately, there is an implicit symmetrization which is forced at the step of applying the homology functor, in the sense that simplicial homology \emph{does not distinguish} between simplices of the form $[x,x']$ and $[x',x]$---both are mapped to the same 1-dimensional linear subspace at the vector space level, so it does not matter if one, and not the other, is present in the simplicial complex. More generally, the different orderings of the vertices of a simplex $\s$ belong to exactly two equivalence classes (marked by a positive and a negative orientation), and upon passing to vector spaces, the two equivalence classes map to $\pm v_\s$, where $v_\s$ is the vector space representative of $\s$. Details about this informal argument can be found in \cite[pg. 27]{munkres-book}. Also note that this issue of symmetrization is not exclusive to the persistent framework; it arises immediately in standard (oriented) simplicial homology.

\subsubsection*{Sensitivity to asymmetry.}
There are different degrees to which a persistent homology method may be sensitive to asymmetry. As a first test for sensitivity to asymmetry, one can take a network and ``reverse" the weights on the edges between a given pair of nodes, and verify that the persistent vector spaces of the original and modified networks are different. The Rips method does not track asymmetry even to the level of simplicial complexes---one can verify that given two copies of a network that differ by an edge reversal, the corresponding Rips filtrations are always equivalent \cite[Remark 19]{dowker-arxiv}. 
On the other hand, there are explicit examples of networks for which such an edge reversal gives rise to a different Dowker filtration \cite[Remark 19]{dowker-arxiv}. Thus the Dowker method tracks asymmetry to the level of simplicial complexes, but \emph{not} to the level of vector spaces, by the discussion above.

One way to track the original asymmetry of a network down to the vector space level, while still remaining rooted in a simplicial construction, would be to consider \emph{ordered simplices} \cite[pg. 76]{munkres-book}. For $p\in \Z_+$, an ordered $p$-simplex over a set $X$ is an ordered $(p+1)$-tuple $(x_0,\ldots,x_p)$, possibly with repeated vertices. Homological constructions can be defined on ordered simplicial complexes in a natural way \cite{turner}, with the following important distinction: in an ordered simplicial complex, different orderings of the same vertex set correspond to different ordered simplices, and these in turn correspond to \emph{linearly independent} vector space representatives upon applying homology. This avoids the symmetrization that occurs with homological constructions on a standard (oriented) simplicial complex, but it comes at a cost: the dimension of the resulting vector space grows very quickly with respect to the cardinality of the underlying set. 
However, steady increases in computational capability have led to a revived interest in this approach. In \cite{dlotko2016topological}, the authors used homology on an ordered simplicial complex called a \emph{directed flag} (or \emph{Rips} or \emph{clique}) \emph{complex} to study a digital reconstruction of a neuronal microcircuit released by the Blue Brain Project (\url{http://bluebrain.epfl.ch/}). Another notable application is in \cite{masulli2016topology}, where the authors use the same construction to study the dynamics of artificial neural networks. 
Even more recently \cite[Theorem 11]{turner}, the notion of a directed flag complex has been cast in the persistent framework. More specifically, it was shown that the method of obtaining a persistent vector space from a filtration of directed flag complexes is \emph{stable} (i.e. robust to noise), and hence a valid method of associating persistent homology signatures to asymmetric objects. However, to our knowledge, there is currently no practical implementation of this theoretical persistent framework.

We study an alternative construction of persistence for asymmetric networks that tracks asymmetry down to the vector space level, in the sense described above. Given a directed network, one may produce a \emph{digraph} (i.e. \emph{directed graph}) \emph{filtration} by scanning a threshold parameter $\d \in \R$ and removing the edges with weight $> \d$. Then, given a valid notion of homology for digraphs, one may attempt to construct a persistent homology for digraphs and prove that this method is stable. In a series of papers released between 2012 and 2014, A. Grigor'yan, Y. Lin, Y. Muranov, and S.T. Yau formalized a notion of homology on digraphs called \emph{path homology} \cite{grigor-arxiv}, as well as a homotopy theory for digraphs that is compatible with path homology \cite{grigor-htpy} and the homotopy theory for undirected graphs proposed in \cite{barcelo2001foundations, babson2006homotopy}. This notion of path homology is the central object of study in our work, and our contribution is to study the notion of persistent homology that arises from this construction. While there may be several notions of homology of a digraph, in \cite[\S1]{grigor-arxiv} the authors argue that path homology has demonstrable advantages over other such notions. 
For example, one approach would be to define the clique complex of the underlying undirected graph, and then consider the homology of the clique complex \cite{ivash}.  However, as noted in \cite{grigor-arxiv}, functorial properties such as the K\"unneth formula may fail for this construction---a cycle graph on four nodes has nontrivial 1-dimensional homology, but the Cartesian product of two such cycle graphs has trivial 2-dimensional homology in this clique complex formulation. On the other hand, the path homology construction of \cite{grigor-arxiv} satisfies a K\"unneth formula. We do not focus on comparing the different constructions of homology for graphs, but we list some desirable properties of path homology:
\begin{enumerate}
\item Nontrivial homologies are possible in all dimensions,
\item Functoriality, e.g. digraph maps induce linear maps between path homologies (essential for a persistence framework),
\item Richer structure at the vector space level than those obtained from simplicial constructions. 
\end{enumerate}

To elucidate this last point, we point the reader to Figure \ref{fig:4-node}, where we illustrate two important \emph{motifs} that appear in systems biology: the bi-parallel and bi-fan motifs \cite{milo2002network}. Even the very general directed flag complex construction is unable to distinguish between these two motifs upon passing to homology, whereas path homology is able to tell these two motifs apart. We consider this particular example to be a ``minimal example" showing the difference between path homology and simplicial constructions. We give a detailed explanation for this situation in the discussion linked to Figure \ref{fig:4-node}, but informally, the distinction occurs because the directed flag complex homology passes through a simplicial setting, whereas path homology remains purely in the algebraic setting. This allows for the appearance of richer structures in path homology, giving it greater ability to distinguish between directed structures.

While distinguishing between path homology and directed flag complex homology, it is important to note properties that the two have in common. Notably, given a vertex set $X$, two vertices $x,x'\in X$, and two edges of the form $x\r x'$ and $x'\r x,$ path homology assigns linearly independent vector space representatives to the two edges. One can verify that this is analogous to the situation described for the directed flag complex, and that the analogy extends for higher dimensions as well. Thus both path homology and directed flag complex homology track asymmetry in the input data down to the level of vector spaces.

\subsection{Our contributions}

We describe a \emph{persistent} framework for path homology (\S\ref{sec:pph}), i.e. we define a \emph{persistent path homology \emph{(PPH)}} method that takes an asymmetric network as input and produces a \emph{path persistence diagram \emph{(PPD)}} as output. We show that this method is stable with respect to a network distance $\dn$ that is analogous to the Gromov-Hausdorff distance between finite metric spaces and has appeared often in recent literature, but also as early as in \cite{clust-net}. By virtue of this stability result, we know that PPH is quantitatively robust to perturbations of input data, and is thus a valid method for data analysis.

In \S\ref{sec:exp} we present the results of testing our method on both simulated and real-world datasets. Our first simulated database is a collection of 16,000 asymmetric networks with weights chosen uniformly at random from the interval $[0,1]$. We computed the 1-dimensional PPD for each network and its transpose, and compared the resulting diagrams via the \emph{bottleneck distance} \cite{dgh-pers}. Guided by the results of these experiments, we were able to conjecture and prove the following:

\begin{theorem}\label{thm:transp-inv} 
PPH is transposition-invariant in each dimension. More specifically, let $(X,A_X)$ be a dissimilarity network, and consider its transpose $(X,A_X^\top)$ (here $A_X^\top$ is the transpose of the matrix representing $A_X$). Then the $k$-dimensional persistent path homologies of $(X,A_X)$ and $(X,A_X^\top)$ are isomorphic, for each $k \in \Z_+$.
\end{theorem}

To draw a comparison with the existing literature, based on experimental observations, we make the following conjecture relating PPH to the Dowker persistent homology method: 

\begin{conjecture}\label{conj:pph-dowk} 
On 3-node networks, 1-dimensional PPH is isomorphic to the 1-dimensional Dowker persistent homology described in \cite{dowker-arxiv}. 
\end{conjecture} 

As a remark related to the preceding conjecture, we observe that there are examples of 4-node networks for which 1-dimensional PPH is distinct from 1-dimensional Dowker persistent homology. One such network is provided in Figure \ref{fig:pph-dowk-4}.

In addition to testing our implementation on a database of random networks, we also tested 1-dimensional PPH on a family of \emph{cycle networks}. A cycle network on $n$ nodes, $n\in \N$, is constructed as follows: take a directed cyclic graph with edge weights equal to 1 in a clockwise direction, and endow it with the shortest path metric. A cycle network on 6 nodes is illustrated in Figure \ref{fig:cyc6}, along with its weight matrix.

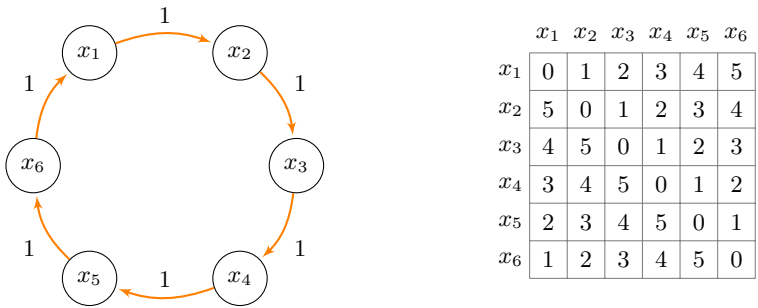
\begin{figure}[h]
\begin{center}
\begin{tikzpicture}[every node/.style={font=\footnotesize}]
\begin{scope}[draw]
\node[circle,draw](1) at (-1,1.5){$x_1$};
\node[circle,draw](2) at (1,1.5){$x_2$};
\node[circle,draw](3) at (1.75,0){$x_3$};
\node[circle,draw](4) at (1,-1.5){$x_4$};
\node[circle,draw](5) at (-1,-1.5){$x_5$};
\node[circle,draw](6) at (-1.75,0){$x_6$};
\end{scope}

\begin{scope}[draw]
\path[myarrow] (1) edge [bend left=20] node[black,above,pos=0.5]{$1$} (2);
\path[myarrow] (2) edge [bend left=20] node[black,above right,pos=0.5]{$1$} (3);
\path[myarrow] (3) edge [bend left=20] node[black,below right,pos=0.5]{$1$} (4);
\path[myarrow] (4) edge [bend left=20] node[black,above,pos=0.5]{$1$} (5);
\path[myarrow] (5) edge [bend left=20] node[black,below left,pos=0.5]{$1$} (6);
\path[myarrow] (6) edge [bend left=20] node[black,above left,pos=0.5]{$1$} (1);
\end{scope}

\begin{scope}[xshift=2.5in]
%\node at (-0.25,-2) {Weight matrix};
\draw[step=0.5cm,color=gray] (-1.5,-1.5) grid (1.5,1.5);

\node at (-1.25,+1.75) {$x_1$};
\node at (-0.75,+1.75) {$x_2$};
\node at (-0.25,+1.75) {$x_3$};
\node at (+0.25,+1.75) {$x_4$};
\node at (+0.75,+1.75) {$x_5$};
\node at (+1.25,+1.75) {$x_6$};

\node at (-1.75,+1.25) {$x_1$};
\node at (-1.75,+0.75) {$x_2$};
\node at (-1.75,+0.25) {$x_3$};
\node at (-1.75,-0.25) {$x_4$};
\node at (-1.75,-0.75) {$x_5$};
\node at (-1.75,-1.25) {$x_6$};

\node at (-1.25,+1.25) {$0$};
\node at (-0.75,+1.25) {$1$};
\node at (-0.25,+1.25) {$2$};
\node at (+0.25,+1.25) {$3$};
\node at (+0.75,+1.25) {$4$};
\node at (+1.25,+1.25) {$5$};

\node at (-1.25,+0.75) {$5$};
\node at (-0.75,+0.75) {$0$};
\node at (-0.25,+0.75) {$1$};
\node at (+0.25,+0.75) {$2$};
\node at (+0.75,+0.75) {$3$};
\node at (+1.25,+0.75) {$4$};

\node at (-1.25,+0.25) {$4$};
\node at (-0.75,+0.25) {$5$};
\node at (-0.25,+0.25) {$0$};
\node at (+0.25,+0.25) {$1$};
\node at (+0.75,+0.25) {$2$};
\node at (+1.25,+0.25) {$3$};

\node at (-1.25,-0.25) {$3$};
\node at (-0.75,-0.25) {$4$};
\node at (-0.25,-0.25) {$5$};
\node at (+0.25,-0.25) {$0$};
\node at (+0.75,-0.25) {$1$};
\node at (+1.25,-0.25) {$2$};

\node at (-1.25,-0.75) {$2$};
\node at (-0.75,-0.75) {$3$};
\node at (-0.25,-0.75) {$4$};
\node at (+0.25,-0.75) {$5$};
\node at (+0.75,-0.75) {$0$};
\node at (+1.25,-0.75) {$1$};

\node at (-1.25,-1.25) {$1$};
\node at (-0.75,-1.25) {$2$};
\node at (-0.25,-1.25) {$3$};
\node at (+0.25,-1.25) {$4$};
\node at (+0.75,-1.25) {$5$};
\node at (+1.25,-1.25) {$0$};
\end{scope}

\end{tikzpicture}
\end{center}
\caption{A cycle network on 6 nodes, along with its weight matrix. Note that the weights are highly asymmetric.}
\label{fig:cyc6}
\end{figure}

Cycle networks form a useful family of examples of asymmetric networks. In \cite[Theorem 22]{dowker-arxiv}, the authors provided a complete characterization of the 1-dimensional Dowker persistent homology of cycle networks on $n$ nodes, for any $n\in \N$. Our experiments suggest that such a characterization result is possible in the case of 1-dimensional PPH as well. In particular, we conjecture the following characterization result:

\begin{conjecture}\label{conj:pph-cycle} 
On cycle networks having any number of nodes, 1-dimensional PPH is isomorphic to 1-dimensional Dowker persistent homology. More specifically, the 1-dimensional PPD of a cycle network on $n$ nodes, $n\in \N$, contains exactly one off-diagonal point $(1,\ceil{n/2})$. 
\end{conjecture}

Beyond testing our method on simulated datasets, we also applied 1-dimensional PPH to a real-world dataset of directed networks arising from 19 years of U.S. economic data (1997-2015). We were interested in seeing if our method was sensitive to significant changes in the economy, such as that caused by the 2007-2008 financial crisis. This is indeed the case, as we show in Figure \ref{fig:econ}. 

A package including our software and datasets will eventually be released on \url{https://research.math.osu.edu/networks/}.

\subsection{Organization of the paper} There are three topics that form the preliminaries of this paper. \S\ref{sec:back-ph} contains the necessary background on persistent homology, \S\ref{sec:back-nets} contains the background on networks and the network distance $\dn$, and \S\ref{sec:back-path} contains the background on digraphs and path homology. In \S\ref{sec:pph} we combine ingredients from the preceding sections to define PPH and prove its stability. In \S\ref{sec:exp} we describe our experiments and explain how they relate to our conjectures.

\subsection{Notation}
We denote the nonempty elements of the power set of a set $X$ by $\pow(X)$, and use the convention that the empty set is excluded from $\pow(X)$. We write $\Z_+, \R_+$ to denote the nonnegative integers and reals, respectively. We will write $\overline{\R}$ to denote the extended real numbers $[-\infty,\infty]$. We fix a field $\mathbb{K}$ and use it throughout the paper. The identity map on a set $X$ is denoted $\id_X$. Given vector spaces $V,V'$, we write $V\cong V'$ to denote isomorphism of vector spaces. Given a finite set $S$, we write $\mb{K}[S]$ to denote the free vector space over $\mb{K}$ generated by the elements of $S$. When we have a sequence of maps $(f_i)_{i\in I}$ indexed by a set $I$, we will often refer to them collectively as $f_\bullet$, without specifying an index. Given sets $A, B$, a map $f:A \r B$, and subsets $S_A \subseteq A, S_B \subseteq B$, we will write $f(S_A) \subseteq S_B$ to mean that $f(s) \in S_B$ for each $s\in S_A$.

\section{Background on Persistent Homology}
\label{sec:back-ph}

Homology is the formal algebraic construction at the center of our work. For our purposes, we define homology in the setting of general vector spaces, and refer the reader to \cite[\S1.13]{munkres-book} for additional details. A \emph{chain complex} is defined to be a sequence of vector
spaces $(C_k)_{k\in \Z}$ and \emph{boundary maps} $(\p_k:C_k\r C_{k-1})_{k\in \Z}$ satisfying the condition $\p_{k-1}\circ \p_k =0$ for each $k \in \Z$. We often denote a chain complex as $\mc{C}=(C_k,\p_k)_{k\in \Z}$. Because our constructions are finite, often there will exist $m, M\in \Z$ such that $C_k \cong \set{0}$ for each integer $k<m$, and $C_k \cong C_j$ for all integers $j,k > M$. In the remainder of this paper, we will use the nonnegative integers $\Z_+$ to index a chain complex, and define $C_{-1}, C_{-2}$ as needed.

Given a chain complex $\mc{C}$ and any $k\in \Z_+$, one may define the following subspaces: 
\begin{align*}
Z_k(\mc{C}) &:= \ker(\p_k) = \set{c\in C_k : \p_kc = 0}, \text{ the \emph{$k$-cycles}},\\ 
B_k(\mc{C}) &:=\im(\p_{k+1}) = \set{c \in C_{k} : c = \p_{k+1} b \text{ for some } b\in C_{k+1}}, \text{the \emph{$k$-boundaries}}.
\end{align*}

The quotient vector space $H_k(\mc{C}) :=Z_k(\mc{C})/B_k(\mc{C}) $ is called the \emph{$k$-th homology vector space of the chain complex $\mc{C}$}. The dimension of $H_k(\mc{C})$ is called the \emph{$k$-th Betti number} of $\mc{C}$, denoted $\b_k(\mc{C})$.

Given two chain complexes $\mc{C}=(C_k,\p_k)_{k\in \Z}$ and $\mc{C'}=(C'_k,\p'_k)_{k\in \Z}$, a \emph{chain map} $\ph: \mc{C} \r \mc{C'}$ is a family of morphisms $(\ph_k : C_k \r C'_k)_{k\in \Z_+}$ such that $\p'_k \circ \ph_k = \ph_{k-1} \circ \p_k$ for each $k\in \Z_+$. One can verify that such a chain map induces a family of linear maps $(\ph_\#)_k: H_k(\mc{C}) \r H_k(\mc{C}')$ for each $k\in \Z_+$ \cite[p. 72]{munkres-book}.

A \emph{persistent vector space} \cite[Definition 3.3]{carlsson-acta} is defined to be a family of vector spaces $\{V^{\d} \xr{\nu_{\d,\d'}} V^{\d'}\}_{\d \leq \d'\in \R}$ with linear maps between them, such that: (1) $\dim(V^{\d}) < \infty$ for each $\d \in \R$, (2) there exist $\d_I, \d_F \in \R$ such that all maps $\nu_{\d,\d'}$ are linear isomorphisms for $\d,\d' \geq \d_F$ and for $\d,\d' \leq \d_I$, (3) there are only finitely many values of $\d \in \R$ such that $V^{\d-\e} \not\cong V^{\d}$ for each $\e>0$, and (4) for any $\d < \d' < \d'' \in \R$, we have $\nu_{\d',\d''}\circ \nu_{\d,\d'} = \nu_{\d,\d''}$.  

Even though the preceding definition uses $\R$ as the indexing set, note that there are only finitely many values of $\d \in \R$ for which the linear maps are not isomorphisms. As a consequence, one may equivalently define a persistent vector space to be an $\N$-indexed family $\{V^{\d_i} \xr{\nu_{i,i+1}} V^{\d_{i+1}})_{i\in \N}$ such that: (1) $\dim(V^{\d_i})< \infty$ for each $i\in \N$, and (2) there exists $F\in \N$ such that all the $\nu_{i,i+1}$ are isomorphisms for $i\geq F$. An explicit equivalence between these two notions is provided in \cite[\S 2.1]{dowker-arxiv}.

\subsection{Persistence diagrams and barcodes} To each persistent vector space, one may associate a multiset of intervals, called a \emph{persistence barcode} or \emph{persistence diagram}. This barcode is a full invariant of a persistent vector space \cite{zomorodian2005computing}, and it has the following natural interpretation: given a barcode corresponding to a persistent vector space obtained from a filtered simplicial complex, the long bars correspond to meaningful topological features, whereas the short bars correspond to noise or artifacts in the data. The standard treatment of persistence barcodes and diagrams appears in \cite{frosini1992measuring,frosini1990distance,robins1999towards, edelsbrunner2002topological}, and \cite{zomorodian2005computing}. We follow a more modern presentation that appeared in \cite{ph-self}. To build intuition, we refer the reader to Figure \ref{fig:pers}.

\begin{figure}[b]
\includegraphics[scale=0.7]{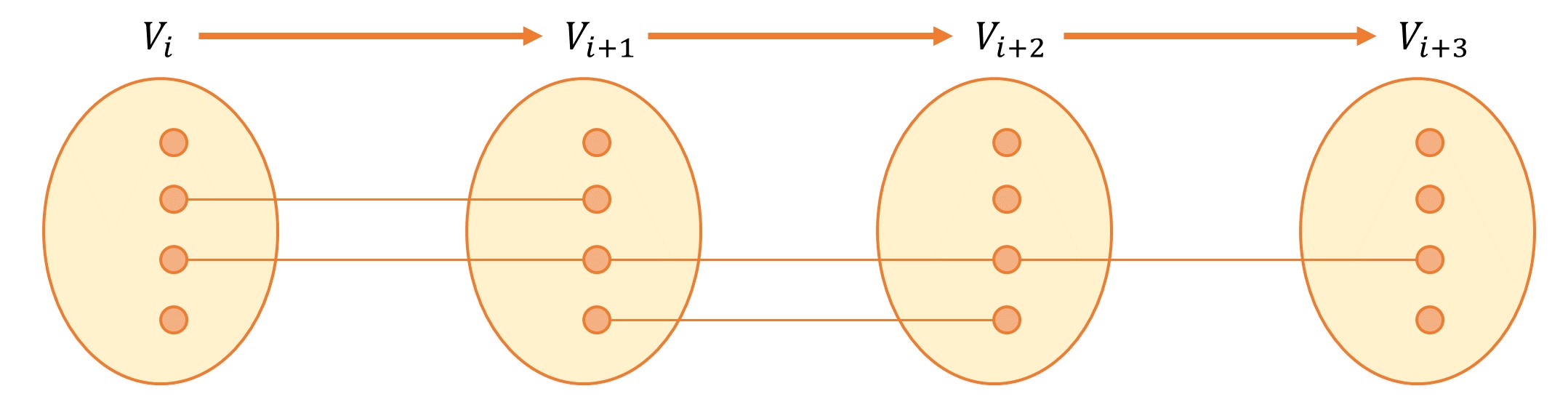}
\caption{Intuition behind a persistence barcode. Let $i\in \N$, and consider a sequence of vector spaces $V_i, V_{i+1}, V_{i+2},V_{i+3}$ as above, with linear maps $\{\nu_{i,i+1},\nu_{i+1,i+2},\nu_{i+2,i+3}\}$. The dark dots represent basis elements, where the bases are chosen such that $\nu_{i,i+1}$ maps the basis elements of $V_i$ to those of $V_{i+1}$, and so on. Such a choice of basis is possible by performing row and column operations on the matrices of the linear maps \cite[Basis Lemma]{ph-self}. The persistence barcode of this sequence can then be read off from the ``strings" joining the dots. In this case, the barcode is the collection $\{[i,i+1],[i,i+3],[i+1,i+2]\}$. Note that when these intervals are read in $\Z$, they are the same as the half-open intervals one would expect from the definition of the persistence barcode given above.}
\label{fig:pers}
\end{figure}

Let $\V = \{V^{\d_i} \xr{\nu_{i,i+1}} V^{\d_{i+1}})_{i\in \N}$ be a persistent vector space. Because all but finitely many of the $\nu_{\bullet,\bullet+1}$ maps are isomorphisms, one may choose \emph{compatible bases} $(B_i)_{i\in \N}$ for each $V^{\d_i}$, $i\in \N$, such that $\nu_{i,i+1}|_{B_i}$ is injective for each $i\in \N$, and  
\[\rank (\nu_{i,i+1}) = \card(\im(\nu_{i,i+1}|_{B_i})\cap B_{i+1}), \text{ for each } i\in \N \text{ \cite[Basis Lemma]{ph-self}}.\] 
Here $\nu_{i,i+1}|_{B_i}$ denotes the restriction of $\nu_{i,i+1}$ to the set $B_i$. Fix such a collection $(B_i)_{i\in \N}$ of bases. Next define:
\[L:=\{(b,i) : b \in B_i, b \not\in \im(\nu_{i-1,i}), i\in \{2,3,4,\ldots\}\} \cup \{(b,1) : b\in B_1\}.\]
For each $(b,i)\in L$, the integer $i$ is interpreted as the \emph{birth index} of the basis element $b$. Informally, it refers to the first index at which a certain algebraic signal appears in the persistent vector space. Next define a map $\ell:L \r \N$ as follows:
\[\ell(b,i):=\max\{k\in \N : (\nu_{k-1,k}\circ \cdots \circ \nu_{i+1,i+2}\circ \nu_{i,i+1})(b) \in B_{k}\}.\]
For each $(b,i)\in L$, the integer $\ell(b,i)$ is often called a \emph{death index}. Informally, it refers to the index at which the signal $b$ disappears from the persistent vector space. 
The \emph{persistence barcode of $\V$} is then defined to be the following multiset of intervals: 
\[\pers(\V):=\big[[\d_i,\d_{j+1}) : \text{there exists } (b,i)\in L \text{ such that } \ell(b,i)=j \big],\]
where the bracket notation denotes taking the multiset and the multiplicity of $[\d_i,\d_{j+1})$ is the number of elements $(b,i)\in L$ such that $\ell(b,i)=j$.

These intervals, which are called \emph{persistence intervals}, are then represented as a set of lines over a single axis. Equivalently, the intervals in $\pers(\V)$ can be visualized as a multiset of points lying on or above the diagonal in $\overline{\R}^2$, counted with multiplicity. This is the case for the \emph{persistence diagram of $\V$}, which is defined as follows:
\[\dgm(\V):=\big[(\d_i,\d_{j+1}) \in \overline{\R}^2 : [\d_i,\d_{j+1}) \in \pers(\V) \big],\]
where the multiplicity of $(\d_i,\d_{j+1})\in \overline{\R}^2$ is given by the multiplicity of $[\d_i,\d_{j+1}) \in \pers(\V)$.

The \emph{bottleneck distance} between persistence diagrams, and more generally between multisets $A,B$ of points in $\overline{\R}^2$, is defined as follows:
\[\db(A,B):= \inf\set{\sup_{a\in A}\norm{a - \ph(a)}_\infty : \ph:A\cup \Delta^\infty \r B\cup \Delta^\infty \text{ a bijection}}.\]
Here $\norm{(p,q)-(p',q')}_\infty:=\max(|p-p'|,|q-q'|)$ for each $p,q,p',q'\in \R$, and $\Delta^\infty$ is the multiset consisting of each point on the diagonal, taken with infinite multiplicity.

\begin{remark}\label{rem:trivial-diag} From the definition of bottleneck distance, it follows that points in a persistence diagram $\dgm(\mc{V})$ that belong to the diagonal do not contribute to the bottleneck distance between $\dgm(\mc{V})$ and another diagram $\dgm(\mc{U})$. Thus whenever we describe a persistence diagram as being \emph{trivial}, we mean that either it is empty, or it does not have any off-diagonal points. 
\end{remark}

There are numerous ways of formulating the definitions we have provided in this section. For more details, we refer the reader to \cite{robins1999towards, edelsbrunner2002topological, zomorodian2005computing, zigzag, edelsbrunner2010computational, bauer-isom, ph-self}.

\subsection{Interleaving distance and stability of persistent vector spaces.}
\label{sec:background-int}

Given $\e \geq 0$, two $\R$-indexed persistent vector spaces $\mc{V}=\{V^\d\xr{\nu_{\d,\d'}} V^{\d'}\}_{\d\leq \d'}$ and $\mc{U}=\{U^\d\xr{\mu_{\d,\d'}} U^{\d'}\}_{\d\leq \d'}$ are said to be \emph{$\e$-interleaved} \cite{chazal2009proximity,bauer-isom} if there exist two families of linear maps
\begin{align*}
\{\ph_{\d}&:V^\d \r V^{\d + \e}\}_{\d \in \R},\\ 
\{\psi_{\d}&:U^\d \r U^{\d + \e}\}_{\d \in \R}
\end{align*}
such that the following diagrams commute for all $\d' \geq \d\in \R$:

\[ \begin{tikzcd}[column sep=large]
V^\d \arrow{r}{\nu_{\d,\d'}} \arrow[swap]{dr}{\ph_\d} & 
V^{\d'}\arrow{dr}{\ph_{\d'}} & 
{} & 
{} & 
V^{\d+\e} \arrow{r}{\nu_{\d+\e,\d'+\e}} &
V^{\d'+\e} \\
{} &
U^{\d+\e} \arrow{r}{\mu_{\d+\e,\d'+\e}} & 
U^{\d'+\e} &
U^{\d} \arrow{ur}{\psi_\d} \arrow{r}{\mu_{\d,\d'}} &
U^{\d'} \arrow[swap]{ur}{\psi_{\d'}} & {}
\end{tikzcd} \]

\[ \begin{tikzcd}[column sep=large]
V^\d \arrow{rr}{\nu_{\d,\d+2\e}} \arrow[swap]{dr}{\ph_\d} & {} &
V^{\d+2\e} & 
{}&
V^{\d+\e} \arrow{dr}{\psi_{\d+\e}}\\
{} &
U^{\d+\e} \arrow[swap]{ur}{\ph_{\d+\e}} & 
{}&
U^{\d} \arrow{ur}{\psi_{\d}} \arrow{rr}{\mu_{\d,\d+2\e}} & {} &
U^{\d+2\e} 
\end{tikzcd} \]

The purpose of introducing $\e$-interleavings is to define a pseudometric on the collection of persistent vector spaces. The \emph{interleaving distance} between two $\R$-indexed persistent vector spaces $\mc{V},\U$ is given by:
\[\di(\U,\V) := \inf \set{\e \geq 0 : \text{$\U$ and $\V$ are $\e$-interleaved}}.\]
One can verify that this definition induces a pseudometric on the collection of persistent vector spaces \cite{chazal2009proximity, bauer-isom}. The interleaving distance can then be related to the bottleneck distance as follows:
\begin{theorem}[Algebraic Stability Theorem, \cite{chazal2009proximity}]\label{thm:alg-stab}
Let $\U, \V$ be two $\R$-indexed persistent vector spaces. Then,
\[\db(\dgm(\U),\dgm(\V))\leq \di(\U,\V).\]
\end{theorem}

\section{Background on Directed Networks and Network Distances}
\label{sec:back-nets}

We follow the framework of \cite{carlsson2013axiomatic,clust-net}. Recall that a \emph{(dissimilarity) network} is a finite set $X$ together with a \emph{weight function} $A_X:X\times X \r \R_+$ such that for any $x,y\in X$: (1) $A_X(x,y)\geq 0$, and (2) $A_X(x,y)=0 \iff x=y$. Note that $A_X$ is not required to satisfy the triangle inequality or any symmetry condition. The collection of all such networks is denoted $\Ncal$.

When comparing networks of the same size, e.g. two networks $(X,A_X),(X,A'_X)$, a natural method is to consider the $\ell^\infty$ distance:
\[\norm{A_X-A'_X}_\infty:=\max_{x,x'\in X}|A_X(x,x')-A'_X(x,x')|.\]

But one would naturally want a generalization of the $\ell^\infty$ distance that works for networks having different sizes. In this case, one needs a way to correlate points in one network with points in the other. To see how this can be done, let $(X,A_X), (Y,A_Y) \in \Ncal$. Let $R$ be any nonempty relation between $X$ and $Y$, i.e. a nonempty subset of $X \times Y$. The \emph{distortion} of the relation $R$ is given by:
\[\dis(R):=\max_{(x,y),(x',y')\in R}|A_X(x,x')-A_Y(y,y')|.\] 

A \emph{correspondence between $X$ and $Y$} is a relation $R$ between $X$ and $Y$ such that $\pi_X(R)=X$ and $\pi_Y(R)=Y$, where $\pi_X:X\times Y \r X$ and $\pi_Y:X\times Y \r Y$ denote the natural projections. The collection of all correspondences between $X$ and $Y$ will be denoted $\Rsc(X,Y)$.

Following prior work in \cite{clust-net}, the \emph{network distance} $\dn:\Ncal \times \Ncal \r \R_+$ is then defined as:
\[\dn(X,Y):=\tfrac{1}{2}\min_{R\in\Rsc}\dis(R).\]

It can be verified that $\dn$ as defined above is a pseudometric, and that the networks at 0-distance can be completely characterized \cite{nets-allerton,nets-icassp}. Next we wish to state a reformulation of $\dn$ that will aid our proofs. First we define the distortion of a map between two networks. Given any $(X,A_X),(Y,A_Y)\in \Ncal$ and a map $\ph:(X,A_X) \r (Y,A_Y)$, the \emph{distortion} of $\ph$ is defined as:
\[\dis(\ph):= \max_{x,x'\in X}|A_X(x,x')-A_Y(\ph(x),\ph(x'))|.\]
Next, given maps $\ph:(X,A_X)\r (Y,A_Y)$ and $\psi:(Y,A_Y)\r (X,A_X)$, we define two \emph{co-distortion} terms: 
\begin{align*}
C_{X,Y}(\ph,\psi) &:= \max_{(x,y)\in X\times Y}|A_X(x,\psi(y)) - A_Y(\ph(x),y)|,\\
C_{Y,X}(\psi,\ph) &:= \max_{(y,x)\in Y\times X}|A_Y(y,\ph(x)) - A_X(\psi(y),x)|.
\end{align*}

\begin{proposition}[{\cite[Proposition 4]{dowker-arxiv}}]
\label{prop:dn-ko}
Let $(X,A_X), (Y,A_Y)\in \Ncal$. Then,
\[ \dn(X,Y) = \tfrac{1}{2}\min\{\max(\dis(\ph),\dis(\psi),C_{X,Y}(\ph,\psi),C_{Y,X}(\psi,\ph)) : \ph:X \r Y, \psi:Y \r X \text{ any maps}\}.\] 
\end{proposition}

\begin{remark} Proposition \ref{prop:dn-ko} is analogous to a result of Kalton and Ostrovskii \cite[Theorem 2.1]{kalton1997distances} which involves the Gromov-Hausdorff distance between metric spaces. In particular, when restricted to the special case of networks that are also metric spaces, the network distance $\dn$ agrees with the Gromov-Hausdorff distance. Details on the Gromov-Hausdorff distance can be found in \cite{burago}. 

An important remark is that in the Kalton-Ostrovskii formulation, there is only one co-distortion term. When Proposition \ref{prop:dn-ko} is applied to metric spaces, the two co-distortion terms become equal by symmetry, and thus the Kalton-Ostrovskii formulation is recovered. But \emph{a priori}, the lack of symmetry in the network setting requires us to consider both of the co-distortion terms.
\end{remark}

\section{Background on Digraphs and Path Homology}
\label{sec:back-path}

In what follows, we summarize and condense some concepts that appeared in \cite{grigor-arxiv}, and attempt to preserve the original notation wherever possible. 

\begin{definition}\label{defn:aeq}
Before proceeding, recall that a \emph{digraph} is a pair $G=(X,E)$, where $X$ is a finite set (the \emph{vertices}) and $E$ is a subset of $X\times X$ (the \emph{edges}). We will always consider digraphs without self-loops. We also make the following remark on notation: given $x,x' \in X$ for a digraph $G=(X,E)$, we will write $x \aeq x'$ to mean:
\begin{align*}
\text{either } x=x',\text{ or } (x,x')\in E.
\end{align*} 
\end{definition}

\subsection{Vector spaces of paths}

Given a finite set $X$ and any integer $p \in Z_+$, an \emph{elementary $p$-path over $X$} is a sequence $[x_0,\ldots, x_p]$ of $p+1$ elements of $X$. For each $p\in \Z_+$, the free vector space consisting of all formal linear combinations of elementary $p$-paths over $X$ with coefficients in $\mathbb{K}$ is denoted $\La_p= \La_p(X)=\La_p(X,\mathbb{K})$. One also defines $\La_{-1}:=\mb{K}$ and $\La_{-2}:=\set{0}$. Next, for any $p \in \Z_+ $, one defines a linear map $\p_p^{\nr}: \La_p \r \La_{p-1}$ to be the linearization of the following map on the generators of $\La_p$:
\[\p_p^{\nr}([x_0,\ldots,x_p]):=\sum_{i=0}^p(-1)^i[x_0,\ldots,\widehat{x_i},\ldots,x_p], \text{ for each elementary $p$-path } [x_0,\ldots,x_p]\in \La_p.\]
Here $\widehat{x_i}$ denotes omission of $x_i$ from the sequence. The maps $\p^{\nr}_\bullet$ are referred to as the \emph{non-regular boundary maps}. For $p=-1$, one defines $\p_{-1}^{\nr}:\La_{-1} \r \La_{-2}$ to be the zero map. One can then verify that $\p_{p+1}^{\nr}\circ \p_p^{\nr} = 0$ for any integer $p \geq -1$ \cite[Lemma 2.2]{grigor-hlgy}. It follows that $(\La_p,\p^{\nr}_p)_{p\in \Z_+}$ is a chain complex.

For notational convenience, we will often drop the square brackets and commas and write paths of the form $[a,b,c]$ as $abc$. We use this convention in the next example.

\begin{example} Let $X=\set{p}$ be a singleton. Then, because there is no restriction on repetition, we have:
\[\La_0(X) = \mb{K}[\{p\}], \quad \La_1(X) = \mb{K}[\{pp\}], \quad \La_2(X) = \mb{K}[\{ppp\}], \text{ and so on.}\]

Next let $Y=\set{a,b}$. Then we have:
\[\La_0(Y) = \mb{K}[\{a,b\}], \quad \La_1(Y) = \mb{K}[\{aa,ab,ba,bb\}], \text{ and so on.}\]
\end{example}

\begin{example}\label{ex:2-node-dnr}
We will soon explain the interaction between paths on a set and the edges on a digraph. To build intuition, first consider a digraph on a vertex set $Y=\{a,b\}$ as in Figure \ref{fig:2-node}:
\begin{figure}[h]
\begin{center}
\begin{tikzpicture}[every node/.style={font=\footnotesize}]
\begin{scope}[draw]
\node[circle,draw](A) at (2,0){$a$};
\node[circle,draw](B) at (5,0){$b$};
\end{scope}

\begin{scope}[draw]
\path[myarrow,draw,thick] (A) edge [bend left] (B);
\path[myarrow,draw,thick] (B) edge [bend left] (A);
\end{scope}

\end{tikzpicture}
\end{center}
\caption{A two-node digraph on the vertex set $Y=\set{a,b}$. Operations of $\p^{\nr}_\bullet$ and $\p_\bullet$ on this digraph are discussed in Examples \ref{ex:2-node-dnr} and \ref{ex:2-node-d}.}
\label{fig:2-node}
\end{figure}
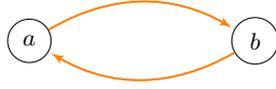

Notice that there is a legitimate ``path" on this digraph of the form $aba$, obtained by following the directions of the edges. But notice that applying $\p_2^{\nr}$ to the 2-path $aba$ yields $\p_2^{\nr}(aba)= ba - aa + ab$, and $aa$ is not a valid path on this particular digraph. To handle situations like this, one needs to consider \emph{regular paths}, which are explained in the next section.

\end{example}

\subsubsection{Regular paths}

For each $p\in \Z_+$, an elementary $p$-path $[x_0, \ldots, x_p]$ is called \emph{regular} if $x_i\neq x_{i+1}$ for each $0\leq i\leq p-1$, and \emph{irregular} otherwise. Then for each $p \in \Z_+$, one defines:
\begin{align*}
\mc{R}_p=\mc{R}_p(X,\mathbb{K})&:=\mb{K}\big[\set{[x_0,\ldots, x_p]: [x_0,\ldots, x_p]\text{ is regular}}\big]\\
\mc{I}_p=\mc{I}_p(X,\mathbb{K})&:=\mb{K}\big[\set{[x_0,\ldots, x_p]: [x_0,\ldots, x_p]\text{ is irregular}}\big].
\end{align*}

One can further verify that $\p_p^{\nr}(\mc{I}_p) \subseteq \mc{I}_{p-1}$ \cite[Lemma 2.6]{grigor-hlgy}, and so $\p_p^{\nr}$ is well-defined on $\La_p/\mc{I}_p$. Since $\mc{R}_p \cong \La_p/\mc{I}_p $ via a natural linear isomorphism, one can define $\p_p:\mc{R}_p \r \mc{R}_{p-1}$ as the pullback of $\p_p^{\nr}$ via this isomorphism \cite[Definition 2.7]{grigor-hlgy}. Then $\p_p$ is referred to as the \emph{regular boundary map} in dimension $p$, where $p\in \Z_+$. Now we obtain a new chain complex $(\mc{R}_p,\p_p)_{p\in \Z_+}$.

\begin{example}\label{ex:2-node-d} Consider again the digraph in Figure \ref{fig:2-node}. Applying the regular boundary map to the 2-path $aba$ yields $\p_2(aba)=ba + ab$. This example illustrates the following general principle: 
\begin{center}
\textit{Irregular paths arising from an application of $\p_\bullet$ are treated as zeros.}
\end{center}

\end{example}

\subsubsection{Paths on digraphs}

We now expand on the notion of paths on a set to discuss paths on a digraph. We follow the intuition developed in Examples \ref{ex:2-node-dnr} and \ref{ex:2-node-d}.

Let $G=(X,E)$ be a digraph. For each $p\in \Z_+$, one defines an elementary $p$-path $[x_0,\ldots, x_p]$ on $X$ to be \emph{allowed} if $(x_i,x_{i+1})\in E$ for each $0\leq i\leq p-1$. For each $p\in \Z_+$, the free vector space on the collection of allowed $p$-paths on $(X,E)$ is denoted $\mc{A}_p=\mc{A}_p(G)=\mc{A}_p(X,E,\mathbb{K})$, and is called the \emph{space of allowed $p$-paths}. One further defines $\mc{A}_{-1}:=\mathbb{K}$ and $\mc{A}_{-2}:=\set{0}$. 

\begin{example}\label{ex:3-node-allowed}
Consider the digraphs $G_M, G_N$ in Figure \ref{fig:3-node}.

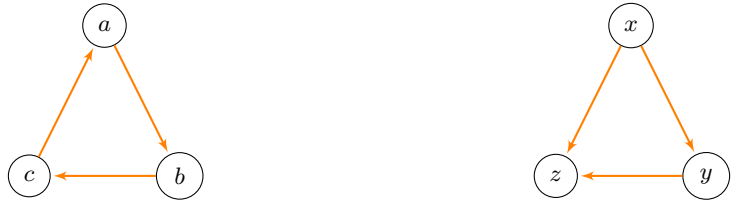
\begin{figure}[h]
\begin{center}
\begin{tikzpicture}[every node/.style={font=\footnotesize}]
\begin{scope}[draw]
\node[circle,draw](A) at (1,2){$a$};
\node[circle,draw](C) at (0,0){$c$};
\node[circle,draw](B) at (2,0){$b$};
\end{scope}

\begin{scope}[draw,xshift=7cm]
\node[circle,draw](X) at (1,2){$x$};
\node[circle,draw](Z) at (0,0){$z$};
\node[circle,draw](Y) at (2,0){$y$};
\end{scope}

\begin{scope}[draw]
\path[myarrow,draw,thick] (A) edge [] (B);
\path[myarrow,draw,thick] (B) edge [] (C);
\path[myarrow,draw,thick] (C) edge [] (A);
\end{scope}

\begin{scope}[draw]
\path[myarrow,draw,thick] (X) edge [] (Y);
\path[myarrow,draw,thick] (Y) edge [] (Z);
\path[myarrow,draw,thick] (X) edge [] (Z);
\end{scope}

\end{tikzpicture}
\end{center}
\caption{Two three-node digraphs $G_M, G_N$ on the vertex sets $M=\set{a,b,c}$ and $N=\set{x,y,z}$. The corresponding vector spaces of allowed paths are described in Example \ref{ex:3-node-allowed}.}
\label{fig:3-node}
\end{figure}
For $0\leq p \leq 3,$ we have the following vector spaces of allowed paths:
\begin{center}
\begin{tabular}{lll}
$\mc{A}_0(G_M)=\mb{K}[\set{a,b,c}]$ & & 
$\mc{A}_0(G_N)=\mb{K}[\set{x,y,z}]$\\
$\mc{A}_1(G_M)=\mb{K}[\set{ab,bc,ca}]$ & & 
$\mc{A}_1(G_N)=\mb{K}[\set{xy,yz,xz}]$\\
$\mc{A}_2(G_M)=\mb{K}[\set{abc,bca,cab}]$ & & 
$\mc{A}_2(G_N)=\mb{K}[\set{xyz}]$\\
$\mc{A}_3(G_M)=\mb{K}[\set{abca,bcab,cabc}]$ &\quad & 
$\mc{A}_3(G_N)=\set{0}$
\end{tabular}
\end{center}

Notice the following interesting situation: applying $\p_2^{G_M}$ to the regular 2-path $abc \in \mc{R}_2(G_M)$ yields $\p_2^{G_M}(abc) = bc - ac + ab \in \mc{R}_1(G_M)$. But whereas $abc$ is an allowed 2-path, $\p_2^{G_M}(abc)$ is \emph{not} allowed, because $ac \not\in \mc{A}_1(G_M)$. So in general, the map $\p_\bullet:\mc{R}_\bullet \r \mc{R}_{\bullet -1}$ does \emph{not} restrict to a map $\mc{A}_\bullet \r \mc{A}_{\bullet-1}$. On the other hand, one can verify that this restriction is valid in the case of $G_N$.  
\end{example}

The situation presented in Example \ref{ex:3-node-allowed} suggests the following natural construction. Given a digraph $G=(X,E)$ and any $p\in \Z_+$, the \emph{space of $\p$-invariant $p$-paths on $G$} is defined to be the following subspace of $\mc{A}_p(G)$:
\[\Om_p=\Om_p(G)=\Om_p(X,E,\mathbb{K}):=\set{ c \in \mc{A}_p: \p_p(c)\in \mc{A}_{p-1}}.\]
One further defines $\Om_{-1}=\mc{A}_{-1}\cong \mb{K}$ and $\Om_{-2}=\mc{A}_{-2}=\set{0}$. 

One can verify that $\im(\p_p(\Om_p))\subseteq \Om_{p-1}$ for any integer $p \geq -1$. Thus we have a chain complex:

\[\ldots \xr{\p_3} \Om_2 \xr{\p_2} \Om_1 \xr{\p_1} \Om_0 \xr{\p_0} \mb{K} \xr{\p_{-1}} 0 \]

For each $p\in \Z_+$, the \emph{$p$-dimensional path homology groups of $G=(X,E)$} are defined as:
\[H_p(G)= H_p(X,E,\mb{K}):=\ker(\p_p)/\im(\p_{p+1}).\]
The elements of $\ker(\p_p)$ are referred to as \emph{$p$-cycles}, and the elements of $\im(\p_{p+1})$ are referred to as \emph{$p$-boundaries}.

\begin{example}\label{ex:4-node-om} We illustrate the construction of $\Om_\bullet$ for the digraphs in Figure \ref{fig:4-node}.  

\begin{figure}[h]
\begin{center}
\begin{tikzpicture}[every node/.style={font=\footnotesize}]
\begin{scope}[draw]
\node[circle,draw](A) at (0,2){$a$};
\node[circle,draw](B) at (2,2){$b$};
\node[circle,draw](C) at (2,0){$c$};
\node[circle,draw](D) at (0,0){$d$};
\end{scope}

\begin{scope}[draw,xshift=7cm]
\node[circle,draw](W) at (0,2){$w$};
\node[circle,draw](X) at (2,2){$x$};
\node[circle,draw](Y) at (2,0){$y$};
\node[circle,draw](Z) at (0,0){$z$};
\end{scope}

\begin{scope}[draw]
\path[myarrow,draw,thick] (A) edge [] (B);
\path[myarrow,draw,thick] (C) edge [] (B);
\path[myarrow,draw,thick] (A) edge [] (D);
\path[myarrow,draw,thick] (C) edge [] (D);
\end{scope}

\begin{scope}[draw]
\path[myarrow,draw,thick] (W) edge [] (X);
\path[myarrow,draw,thick] (X) edge [] (Y);
\path[myarrow,draw,thick] (W) edge [] (Z);
\path[myarrow,draw,thick] (Z) edge [] (Y);
\end{scope}

\end{tikzpicture}
\end{center}
\caption{Two four-node digraphs $G_M, G_N$ on the vertex sets $M=\set{a,b,c,d}$ and $N=\set{w,x,y,z}$. The vector spaces $\Om_\bullet$ for each of these digraphs is discussed in Example \ref{ex:4-node-om}. In systems biology, $G_M, G_N$ are referred to as \emph{bi-fan} and \emph{bi-parallel motifs}, respectively \cite{milo2002network}. As we demonstrate in Example \ref{ex:4-node-om}, path homology is able to distinguish between these two motifs. We also show in Remark \ref{rem:4-node-comparison} that directed flag complex homology cannot tell these motifs apart.}
\label{fig:4-node}
\end{figure}
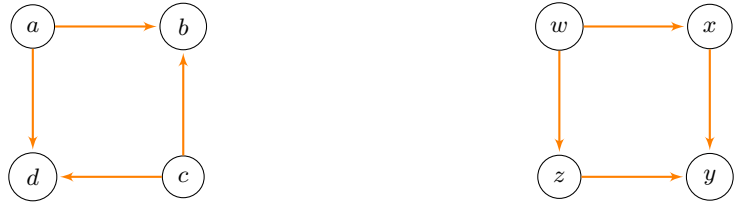

For $0\leq p \leq 2,$ we have the following vector spaces of $\p$-invariant paths:
\begin{center}
\begin{tabular}{lll}
$\Om_0(G_M)=\mb{K}[\set{a,b,c,d}]$ & & 
$\Om_0(G_N)=\mb{K}[\set{w,x,y,z}]$\\
$\Om_1(G_M)=\mb{K}[\set{ab,cb,cd,ad}]$ & & 
$\Om_1(G_N)=\mb{K}[\set{wx,xy,zy,wz}]$\\
$\Om_2(G_M)=\set{0}$ & & 
$\Om_2(G_N)=\mb{K}[\set{wxy-wzy}]$
\end{tabular}
\end{center}

The crux of the $\Om_\bullet$ construction lies in understanding $\Om_2(G_N)$. Note that even though $\p_2^{G_N}(wxy)$, $\p_2^{G_N}(wzy) \not\in \mc{A}_2(G_N)$ (because $wy \not\in \mc{A}_1(G_N)$), we still have:
\[\p_2^{G_N}(wxy-wzy) = xy-wy+wx - zy + wy - wz \in \mc{A}_1(G_N).\]

Let us now determine the 1-dimensional path homologies of $G_M$ and $G_N$. Observe that:
\begin{align*}
\p_1^{G_M}(ab -cb + cd -ad) &= b-a -b+c +d -c -d+a = 0,\\
\p_1^{G_N}(wx +xy - zy -wz) &= x-w +y-x -y +z -z+w = 0.
\end{align*}

One can then verify that 
\[\ker(\p_1^{G_M}) = \mb{K}[\{ab -cb + cd -ad\}] \neq \set{0} = \im(\p_2^{G_M}),\]
and also that 
\[\ker(\p_1^{G_N}) = \mb{K}[\{wx +xy - zy -wz\}] = \im(\p_2^{G_N}).\]

As a consequence, we obtain $\dim(H_1(G_M)) = 1$, and $\dim(H_1(G_N)) = 0$. Thus path homology can successfully distinguish between these two motifs.

\end{example}

\begin{remark}\label{rem:4-node-comparison} 
It is interesting to compare the path homologies just computed with the directed flag complex homology studied in \cite{dlotko2016topological, masulli2016topology, turner}, and to verify that the latter is \emph{not} able to distinguish between the bi-parallel and bi-fan motifs. Given a digraph $G=(X,E)$, the directed flag complex is defined to be the \emph{ordered} simplicial complex given by writing:
\[\mf{F}_G:=X \cup \set{(x_0,\ldots,x_p) : (x_i,x_j) \in E \text{ for all } 0\leq i < j \leq p}.\]
Here we use parentheses to denote ordered simplices. For the motifs in Figure \ref{fig:4-node}, we have:
\[\mf{F}_{G_M} = \set{a,b,c,d,ab,cb,cd,ad} \qquad \text{and} \qquad
\mf{F}_{G_N} = \set{w,x,y,z,wx,xy,wz,zy},\]
where we have omitted parentheses and commas to write terms of the form $(a,b)$ as $ab$. In particular, terms of the form $wxy, wzy$ are absent from $\mf{F}_{G_N}$. Due to the absence of these 2-simplices, one obtains---by computations similar to the ones presented for the path homology setting---that $\dim(H_1^{\mf{F}}(G_M)) = 1 = \dim(H_1^{\mf{F}}(G_N)).$ Here we have written $H_1^{\mf{F}}$ to signify applying homology to the respective directed flag complex. 
\end{remark}

\subsection{Digraph maps and functoriality}
\label{sec:dig-maps}

A \emph{digraph map} between two digraphs $G_X=(X,E_X)$ and $G_Y=(Y,E_Y)$ is a map $f:X \r Y$ such that for any edge $(x,x')\in E_X$, we have $f(x) \aeq f(x')$. Recall from Definition \ref{defn:aeq} that this means: 
\[\text{either }
f(x)=f(x'), \text{ or }
(f(x),f(x'))  \in E_Y.\]

To extend path homology constructions to a persistent framework, we need to verify the \emph{functoriality} of path homology. As a first step, one must understand how digraph maps transform into maps between vector spaces. 

Let $X,Y$ be two sets, and let $f:X\r Y$ be a set map. For each dimension $p\in \Z_+$, one defines a map $(f_* )_p: \La_p(X) \r \La_p(Y)$ to be the linearization of the following map on generators: for any generator $[x_0,\ldots, x_p]\in \La_p(X)$,
\[(f_*)_p([x_0,\ldots,x_p]):=[f(x_0),f(x_1),\ldots,f(x_p)]. \]
Note also that for any $p\in Z_+$ and any generator $[x_0,\ldots, x_p]\in \La_p(X)$, we have:
\begin{align*}
\big((f_*)_{p-1} \circ \p_p^{\nr}\big) ([x_0,\ldots,x_p]) &=\sum_{i=0}^p(-1)^i (f_*)_{p-1}\big([x_0,\ldots,\widehat{x_i},\ldots,x_p]\big)\\
&=\sum_{i=0}^p(-1)^i [f(x_0),\ldots,\widehat{f(x_i)},\ldots,f(x_p)]\\
&=\big(\p_p^{\nr}\circ (f_*)_p \big)([x_0,\ldots,x_p]).
\end{align*}
It follows that $f_*:=((f_*)_p)_{p\in \Z_+}$ is a chain map from $(\La_p(X),\p^{\nr}_p)_{p\in \Z_+}$ to $(\La_p(Y),\p^{\nr}_p)_{p\in \Z_+}$.

Let $p\in \Z_+$. Note that $(f_*)_p(\mc{I}_p(X)) \subseteq \mc{I}_p(Y)$, so $(f_*)_p$ descends to a map on quotients
\[(\widetilde{f}_*)_p: \La_p(X)/\mc{I}_p(X) \r \La_p(Y)/\mc{I}_p(Y)\] which is well-defined. For convenience, we will abuse notation to denote the map on quotients by $(f_*)_p$ as well. Thus we obtain an induced map $(f_*)_p: \mc{R}_p(X) \r \mc{R}_p(Y)$. Since $p\in \Z_+$ was arbitrary, we get that $f_*$ is a chain map from $(\mc{R}_p(X),\p_p)_{p\in \Z_+}$ to $(\mc{R}_p(Y),\p_p)_{p\in \Z_+}$. The operation of this chain map is as follows: for each $p\in Z_+$ and any generator $[x_0,\ldots, x_p]\in \mc{R}_p(X)$,

\[(f_*)_p([x_0,\ldots,x_p]):=\begin{cases}
[f(x_0),f(x_1),\ldots,f(x_p)] &: f(x_0),f(x_1),\ldots,f(x_p) \text{ are all distinct, and}\\
0 &: \text{otherwise.}
\end{cases}\]

We refer to $f_*$ as \emph{the chain map induced by the set map $f: X \r Y$.} 

Now given two digraphs $G_X=(X,E_X)$, $G_Y=(Y,E_Y)$ and a digraph map $f: G_X \r G_Y$, one may use the underlying set map $f:X \r Y$ to induce a chain map $f_*: \mc{R}_\bullet(X) \r \mc{R}_\bullet(Y)$. As one could hope, the restriction of the chain map $f_*$ to the chain complex of $\p$-invariant paths on $G_X$ maps into the chain complex of $\p$-invariant paths on $G_Y$, and moreover, is a chain map. We state this result as a proposition below, and provide a reference for the proof.

\begin{proposition}[Theorem 2.10, \cite{grigor-htpy}]  \label{prop:chain-map}
Let $G_X=(X,E_X), G_Y=(Y,E_Y)$ be two digraphs, and let $f:G_X \r G_Y$ be a digraph map. Let $f_*: \mc{R}_\bullet(X) \r \mc{R}_\bullet(Y)$ denote the chain map induced by the underlying set map $f:X \r Y$. Let $(\Om_p(G_X),\p^{G_X}_p)_{p\in \Z_+}$, $(\Om_p(G_Y),\p^{G_Y}_p)_{p\in \Z_+}$ denote the chain complexes of the $\p$-invariant paths associated to each of these digraphs. Then $(f_*)_p(\Om_p(G_X)) \subseteq \Om_p(G_Y)$ for each $p\in \Z_+$, and the restriction of $f_*$ to $\Om_\bullet(G_X)$ is a chain map. 
\end{proposition}

Henceforth, given two digraphs $G,G'$ and a digraph map $f:G \r G'$, we refer to the chain map $f_*$ given by Proposition \ref{prop:chain-map} as the \emph{chain map induced by the digraph map $f$.} Because $f_*$ is a chain map, we then obtain an induced linear map $(f_\#)_p: H_p(G) \r H_p(G')$ for each $p \in \Z_+$.

Having set up the necessary concepts, we now proceed to show that path homology is functorial.

\begin{proposition}[Functoriality of path homology]
\label{prop:functoriality-ph}
Let $G,G',G''$ be three digraphs.
\begin{enumerate}
\item Let $\id_G:G\r G$ be the identity digraph map. Then $(\id_{G \#})_p:H_p(G) \r H_p(G)$ is the identity linear map for each $p\in \Z_+$. 
\item Let $f: G\r G', g:G'\r G''$ be digraph maps. Then $((g \circ f)_\#)_p = (g_\#)_p \circ (f_\#)_p$ for any $p\in \Z_+$.
\end{enumerate}

\end{proposition}
\begin{proof} Let $p\in \Z_+$. In each case, it suffices to verify the operations on generators of $\Om_p(G)$. Let $[x_0,\ldots, x_p] \in \Om_p(G).$ We will write $\id_{G*}$ to denote the chain map induced by the digraph map $\id_G$. First note that
\[(\id_{G *})_p([x_0,\ldots, x_p]) = [\id_G(x_0),\ldots,\id_G(x_p)] = [x_0,\ldots,x_p].\]
It follows that $(\id_{G*})_p$ is the identity linear map on $\Om_p(G)$, and thus $(\id_{G\#})_p$ is the identity linear map on $H_p(G)$. For the second claim, suppose first that $g(f(x_0)),\ldots,g(f(x_p))$ are all distinct. This implies that $f(x_0),\ldots,f(x_p)$ are also all distinct, and we observe:
\begin{align*}
((g\circ f)_*)_p([x_0,\ldots,x_p])&=
[g(f(x_0)),\ldots,g(f(x_p))] &&\text{assuming $g(f(x_i))$ all distinct}\\
&=
(g_*)_p([f(x_0),\ldots, f(x_p)]) &&\text{because $f(x_i)$ all distinct} \\
&=(g_*)_p\big((f_*)_p([x_0,\ldots,x_p])\big).
\end{align*}
Next suppose that for some $0\leq i\neq j\leq p$, we have $g(f(x_i))=g(f(x_j))$. Then we obtain:
\begin{align*}
((g\circ f)_*)_p([x_0,\ldots,x_p])=0=
(g_*)_p\big((f_*)_p([x_0,\ldots,x_p])\big).
\end{align*}
It follows that $((g\circ f)_*)_p=(g_*)_p \circ (f_*)_p$. The statement of the proposition now follows.\end{proof}

After discussing digraph maps, a natural question to ask is: ``Which digraph maps induce isomorphisms on path homology?" It turns out that one may define a notion of \emph{homotopy of digraphs}, which has the desirable property that homotopy equivalent digraphs have isomorphic path homologies. This is the content of the next section, and is a crucial ingredient for our proof of the stability of persistent path homology.

\subsection{Homotopy of digraphs}

Let $G_X=(X,E_X), G_Y=(Y,E_Y)$ be two digraphs. The \emph{product digraph} $G_X\times G_Y=(X\times Y,E_{X\times Y})$ is defined as follows: 
\begin{align*}
X \times Y &:= \set{(x,y) : x\in X, y\in Y}, \text{ and }\\
E_{X\times Y}&:=\set{((x,y),(x',y'))\in (X\times Y)^2 : x=x' \text{ and } (y,y')\in E_Y, \text{ or } y=y' \text{ and } (x,x')\in E_X}.
\end{align*}

Next, the \emph{line digraphs} $I^+$ and $I^-$ are defined to be the two-point digraphs with vertices $\set{0,1}$ and edges $(0,1)$ and $(1,0)$, respectively.

Two digraph maps $f,g : G_X \r G_Y$ are \emph{one-step homotopic} if there exists a digraph map $F:G_X\times I \r G_Y$, where $I\in \set{I^+,I^-}$, such that:
\[F|_{G_X\times \set{0}}=f \text{ and } F|_{G_X\times \set{1}}=g.\]
Observe that this condition is equivalent to requiring (recall the definition of $\aeq$ from Definition \ref{defn:aeq}):
\begin{equation}
f(x) \aeq g(x) \text{ for all } x\in X, \text{ or } 
g(x) \aeq f(x) \text{ for all } x\in X.\label{eq:htpy}
\end{equation}

Moreover, $f$ and $g$ are \emph{homotopic}, denoted $f\simeq g$, if there is a finite sequence of digraph maps $f_0=f,f_1,\ldots, f_n=g: G_X \r G_Y$ such that $f_i, f_{i+1}$ are one-step homotopic for each $0\leq i \leq n-1$. The digraphs $G_X$ and $G_Y$ are \emph{homotopy equivalent} if there exist digraph maps $f:G_X\r G_Y$ and $g:G_Y \r G_X$ such that $g\circ f \simeq \id_{G_X}$ and $f \circ g \simeq \id_{G_Y}$.

The concept of homotopy yields the following theorem on path homology groups:
\begin{theorem}[Theorem 3.3, \cite{grigor-htpy}]
\label{thm:htpy}
Let $G, G'$ be two digraphs.
\begin{enumerate}
\item Let $f,g:G \r G'$ be two homotopic digraph maps. Then these maps induce identical maps on homology vector spaces. More precisely, the following maps are identical for each $p\in \Z_+$:
\[(f_\#)_p:H_p(G)\r H_p(G') \qquad\qquad 
(g_\#)_p:H_p(G) \r H_p(G').\]
\item If $G$ and $G'$ are homotopy equivalent, then $H_p(G)\cong H_p(G')$ for each $p\in \Z_+$.
\end{enumerate}
\end{theorem}

The main construction of this paper---a \emph{persistent} framework for path homology and a proof of its stability---rely crucially on the preceding theorem. We are now ready to formulate our result.

\section{The Persistent Path Homology of a Network}
\label{sec:pph}

Let $(X,A_X) \in \Ncal$. For any $\d \in \R_+$, the digraph $\G^\d_X=(X,E^\d_X)$ is defined as follows:
\[E^\d_X:=\set{(x,x')\in X\times X: x\neq x',  A_X(x,x') \leq \d}.\]
Note that for any $\d'\geq \d \in \R_+$, we have a natural inclusion map $\G^\d_X \hr \G^{\d'}_X$. Thus we may associate to $(X,A_X)$ the \emph{digraph filtration} $\{\G^\d_X \hr \G^{\d'}_X\}_{\d\leq \d' \in \R_+}$.

The functoriality of the path homology construction (Proposition \ref{prop:functoriality-ph}) enables us to obtain a persistent vector space from a digraph filtration. Thus we make the following definition:
\begin{definition}
Let $(X,A_X)\in \Ncal$, and consider the digraph filtration $\{\G_X^\d \xhr{\iota_{\d,\d'}} \G_X^{\d'}\}_{\d\leq \d' \in \R_+}$. Then for each $p\in \Z_+$, we define the \emph{$p$-dimensional persistent path homology of $(X,A_X)$} to be the following persistent vector space:
\[\mc{H}_p(X,A_X):=\{H_p(\G_X^{\d}) \xr{(\iota_{\d,\d'})_\#} H_p(\G_X^{\d'})\}_{\d\leq \d'\in \R_+}.\]
We then define the \emph{p-dimensional path persistence diagram \emph{(PPD)} of $(X,A_X)$} to be the persistence diagram of $\mc{H}_p(X,A_X)$, denoted $\dgm_p(X,A_X)$.
\end{definition}

The main theorem of this section, which shows that the persistent path homology construction is stable to perturbations of input data, and hence amenable to data analysis, follows below:

\begin{theorem}[Stability] Let $(X,A_X), (Y,A_Y) \in \Ncal$. Let $p\in \Z_+$. Then, 
\[\db(\dgm_p(X,A_X),\dgm_p(Y,A_Y)) \leq 2\dn((X,A_X),(Y,A_Y)).\]
\end{theorem}
\begin{proof}
Let $\eta=2\dn((X,A_X),(Y,A_Y))$. By virtue of Proposition \ref{prop:dn-ko}, we obtain maps $\ph:X \r Y$ and $\psi:Y \r X$ such that $\dis(\ph)\leq \eta, \dis(\psi)\leq \eta$, and $C_{X,Y}(\ph,\psi), C_{Y,X}(\psi,\ph) \leq \eta$. 

\begin{claim} For each $\d \in \R_+$, the maps $\ph,\psi$ induce digraph maps as follows:
\begin{equation*}
\begin{aligned}[c]
\ph_\d: \G_X^{\d} &\r \G_Y^{\d+\eta}\\
X\ni x &\mapsto \ph(x)\in Y 
\end{aligned}
\qquad\qquad
\begin{aligned}[c]
\psi_\d: \G_Y^{\d} &\r \G_X^{\d+\eta}\\
Y\ni y &\mapsto \psi(y)\in X. 
\end{aligned}
\end{equation*}
\end{claim}
\begin{subproof}
Let $\d\in \R_+$, and let $(x,x')\in E_X^{\d}$. Then $A_X(x,x')\leq \d$. Because $\dis(\ph) \leq \eta$, we have $A_Y(\ph(x),\ph(x'))\leq \d+\eta$. Thus $(\ph(x),\ph(x'))\in E_Y^{\d+\eta}$, and so $\ph_\d$ is a digraph map. Similarly, $\psi_\d$ is a digraph map. Since $\d\in \R_+$ was arbitrary, the claim now follows. \end{subproof}

\begin{claim} Let $\d \leq \d'\in \R_+$, and let $s_{\d,\d'}, t_{\d+\eta,\d'+\eta}$ denote the digraph inclusion maps $\G_X^{\d} \hr \G_X^{\d'}$ and $\G_Y^{\d+\eta} \hr \G_Y^{\d'+\eta}$, respectively. Consider the following diagram of digraphs and digraph maps:
\[ \begin{tikzcd}[column sep=large]
\G^\d_X \arrow{r}{s_{\d,\d'}} \arrow[swap]{dr}{\ph_\d} & \G^{\d'}_X\arrow{dr}{\ph_{\d'}} & \\
&\G^{\d+\eta}_Y \arrow{r}{t_{\d+\eta,\d'+\eta}} & \G^{\d'+\eta}_Y
\end{tikzcd} \]
Then $\ph_{\d'}\circ s_{\d,\d'}$ and $t_{\d+\eta,\d'+\eta}\circ \ph_\d$ are homotopic, in particular one-step homotopic. 
\end{claim}
\begin{subproof} Let $x\in X$. We wish to show $\ph_{\d'}(s_{\d,\d'}(x))\aeq t_{\d+\eta,\d'+\eta}(\ph_\d(x))$. But notice that:
\[\ph_{\d'}(s_{\d,\d'}(x)) = \ph_{\d'}(x) = \ph(x),\]
where the second equality is by definition of $\ph_{\d'}$ and the first equality occurs because $s_{\d,\d'}$ is the inclusion map. Similarly, $t_{\d+\eta,\d'+\eta}(\ph_\d(x)) = t_{\d+\eta,\d'+\eta}(\ph(x))=\ph(x).$ Thus we obtain $\ph_{\d'}(s_{\d,\d'}(x))\aeq t_{\d+\eta,\d'+\eta}(\ph_\d(x))$. Since $x$ was arbitrary, it follows that $\ph_{\d'}\circ s_{\d,\d'}$ and $t_{\d+\eta,\d'+\eta}\circ \ph_\d$ are one-step homotopic. \end{subproof}

\begin{claim} Let $\d \in \R$, and let $s_{\d,\d+2\eta}$ denote the digraph inclusion map $\G_X^{\d} \hr \G_X^{\d+2\eta}$. Consider the following diagram of digraph maps:
\[ \begin{tikzcd}
\G_X^\d \arrow{rr}{s_{\d,\d+2\eta}} \arrow[swap]{dr}{\ph_{\d}} & {} &
\G_X^{\d+2\eta}\\
{} &
\G_Y^{\d+\eta} \arrow[swap]{ur}{\psi_{\d+\eta}} & 
{} 
\end{tikzcd} \]
Then $s_{\d,\d+2\eta}$ and $\psi_{\d+\eta} \circ \ph_\d$ are one-step homotopic.
\begin{subproof} Recall that $C_{X,Y}(\ph,\psi)\leq \eta$, which means that for any $x\in X,y\in Y$, we have:
\[|A_X(x,\psi(y))-A_Y(\ph(x),y))|\leq \eta.\]
Let $x\in X$, and let $y=\ph(x)$. Notice that $s_{\d,\d+2\eta}(x) = x$ and $\psi_{\d+\eta}(\ph_{\d}(x))=\psi(\ph(x))$. Also note:
\[A_X(x,\psi(\ph(x))) \leq \eta + A_Y(\ph(x),\ph(x))=\eta \leq \d+2\eta.\]
Thus $s_{\d,\d+2\eta}(x)\aeq \psi_{\d+\eta}(\ph_\d(x))$, and this holds for any $x\in X$. The claim follows. \end{subproof}
\end{claim}

By combining the preceding claims and Theorem \ref{thm:htpy}, we obtain the following, for each $p\in \Z_+$:
\[((s_{\d,\d+2\eta})_\#)_p = ((\psi_{\d+\eta} \circ \ph_\d)_\#)_p, \qquad \qquad
((\ph_{\d'}\circ s_{\d,\d'})_\#)_p = ((t_{\d+\eta,\d'+\eta}\circ \ph_\d)_\#)_p.\]
By invoking functoriality of path homology (Proposition \ref{prop:functoriality-ph}), we obtain:
\[((s_{\d,\d+2\eta})_\#)_p = ((\psi_{\d+\eta})_\#)_p \circ ((\ph_\d)_\#)_p, \qquad \qquad
((\ph_{\d'})_\#)_p \circ (s_{\d,\d'})_\#)_p = ((t_{\d+\eta,\d'+\eta})_\#)_p \circ ((\ph_\d)_\#)_p.\]
By using similar arguments, we can also obtain, for each $p\in \Z_+$,
\[((t_{\d,\d+2\eta})_\#)_p = ((\ph_{\d+\eta})_\#)_p \circ ((\psi_\d)_\#)_p, \qquad \qquad
((\psi_{\d'})_\#)_p \circ (t_{\d,\d'})_\#)_p = ((s_{\d+\eta,\d'+\eta})_\#)_p\circ ((\psi_\d)_\#)_p.\]
Thus $\mc{H}_p(X,A_X)$ and $\mc{H}_p(Y,A_Y)$ are $\eta$-interleaved, for each $p\in \Z_+$. Stability now follows by an application of Theorem \ref{thm:alg-stab}. \end{proof}

\section{Experiments}
\label{sec:exp}

We now present the results from four experiments where we computed 1-dimensional PPDs of simulated and real-world networks. All persistence computations were carried out in Matlab, using $\K=\R$ as the field of coefficients. In a forthcoming version of this paper, we will release the software and datasets used for our computations, as well as details on additional experiments.

\subsection{Transposition invariance} For each natural number $n \in [3,10]$, we generated 2000 dissimilarity networks with asymmetric weights chosen uniformly at random from the interval $[0,1]$. For each of these 16,000 directed networks, we computed both the 1-dimensional PPD of the original network, as well as the 1-dimensional PPD of the network with all the arrows reversed, i.e. the transpose of the underlying weight matrix. We then computed, for each network, the bottleneck distance between the 1-dimensional PPD of the original network and the 1-dimensional PPD of the transposed network. In each case, the bottleneck distance was equal to 0. Based on this experiment, we hypothesized that transposing a network has no effect on its 1-dimensional PPD. Guided by this computational insight, we were able to prove an even stronger statement: transposing a network has no effect on its $k$-dimensional PPD, for any $k\in \Z_+$. This is the content of Theorem \ref{thm:transp-inv}.

\subsection{Relationship between PPDs and Dowker persistence diagrams}
It was shown in \cite{dowker-arxiv} that the 1-dimensional \emph{Dowker persistence diagram} of a network is sensitive to asymmetry. Thus we were interested in the relationship between the 1-dimensional Dowker persistence diagram and the 1-dimensional PPD of a network. We found explicit examples of 4-node networks for which the two differ; one example is provided in Figure \ref{fig:pph-dowk-4}. But for each network in our database of 2000 random networks on 3 nodes as described above, the 1-dimensional Dowker persistence diagram agreed with the 1-dimensional PPD. This motivated us to formulate Conjecture \ref{conj:pph-dowk}.

\begin{figure}[h]
\begin{center}
\begin{tikzpicture}[every node/.style={font=\footnotesize}]
\begin{scope}[draw]
\node[circle,draw](1) at (-1.00,+1.50){$x_1$};
\node[circle,draw](2) at (+1.00,+1.75){$x_3$};
\node[circle,draw](3) at (+1.05,-0.25){$x_2$};
\node[circle,draw](4) at (-1.00,-0.50){$x_4$};
\end{scope}

\begin{scope}[draw]
\path[myarrow] (2) edge [] node[black,above,pos=0.4]{$2$} (1);
\path[myarrow] (2) edge [] node[black,right ,pos=0.4]{$4$} (3);
\path[myarrow] (3) edge [] node[black,below ,pos=0.4]{$1$} (4);
\path[myarrow] (4) edge [] node[black,right ,pos=0.6]{$3$} (1);

\end{scope}

\begin{scope}[xshift = 2 in, draw]
\node[circle,draw](1) at (-1.00,+1.50){$x_1$};
\node[circle,draw](2) at (+1.00,+1.75){$x_3$};
\node[circle,draw](3) at (+1.05,-0.25){$x_2$};
\node[circle,draw](4) at (-1.00,-0.50){$x_4$};

\path[myarrow] (2) edge [] node[black,above,pos=0.4]{$2$} (1);
\path[myarrow] (2) edge [] node[black,right ,pos=0.4]{$4$} (3);
\path[myarrow] (3) edge [] node[black,below ,pos=0.4]{$1$} (4);
\path[myarrow] (4) edge [] node[black,right ,pos=0.6]{$3$} (1);
\path[myarrow] (1) edge [bend right=40] node[black, left,pos=0.4]{$5$} (4);

\end{scope}

\begin{scope}[xshift=-1.5in]
\draw[step=0.5cm,color=gray] (-1.5,-0.5) grid (0.5,1.5);

\node at (-1.25,+1.75) {$x_1$};
\node at (-0.75,+1.75) {$x_2$};
\node at (-0.25,+1.75) {$x_3$};
\node at (+0.25,+1.75) {$x_4$};

\node at (-1.75,+1.25) {$x_1$};
\node at (-1.75,+0.75) {$x_2$};
\node at (-1.75,+0.25) {$x_3$};
\node at (-1.75,-0.25) {$x_4$};

\node at (-1.25,+1.25) {$0$};
\node at (-0.75,+1.25) {$9$};
\node at (-0.25,+1.25) {$11$};
\node at (+0.25,+1.25) {$5$};

\node at (-1.25,+0.75) {$8$};
\node at (-0.75,+0.75) {$0$};
\node at (-0.25,+0.75) {$6$};
\node at (+0.25,+0.75) {$1$};

\node at (-1.25,+0.25) {$2$};
\node at (-0.75,+0.25) {$4$};
\node at (-0.25,+0.25) {$0$};
\node at (+0.25,+0.25) {$10$};

\node at (-1.25,-0.25) {$3$};
\node at (-0.75,-0.25) {$7$};
\node at (-0.25,-0.25) {$6$};
\node at (+0.25,-0.25) {$0$};

\end{scope}

\end{tikzpicture}
\end{center}
\caption{\textbf{Left:} A network on four nodes for which the 1-dimensional path and Dowker persistence diagrams are not equal. The 1-dimensional PPD consists of a single point $(4,5)$, and the 1-dimensional Dowker persistence diagram consists of a single point $(4,6)$. \textbf{Middle, Right:} We illustrate the digraphs that appears at thresholds $\d=4$ and $\d=5$, respectively. Notice that the digraph on the right contains a copy of a bi-parallel motif, which we have already shown to have trivial 1-dimensional path homology in Example \ref{ex:4-node-om}.}
\label{fig:pph-dowk-4}
\end{figure}
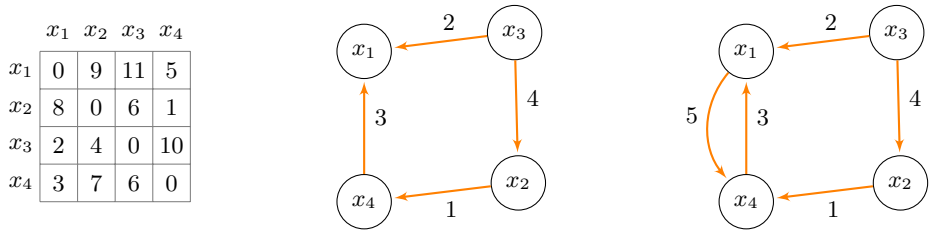

\subsection{PPDs of cycle networks}
For our next experiment, we consider a family of asymmetric networks called \emph{cycle networks}. The construction proceeds as follows. For each $n \in \N$, one considers the weighted, directed graph $(X_n,E_n,W_{E_n})$ with vertex set $X_n:=\set{x_1,x_2,\ldots,x_n}$, edge set $E_n:=\set{(x_1,x_2),(x_2,x_3),\ldots,(x_{n-1},x_n),(x_n,x_1)}$, and edge weights $W_{E_n}:E_n \r \R$ given by writing $W_{E_n}(e)=1$ for each $e\in E_n$. Next let $A_{G_n}:X_n\times X_n \r \R$ denote the shortest path distance induced on $X_n\times X_n$ by $W_{E_n}$. Then one defines the \emph{cycle network on $n$ nodes} as $G_n:=(X_n,A_{G_n})$. An example of a cycle network is illustrated in Figure \ref{fig:cyc6}.

In \cite[Theorem 22]{dowker-arxiv}, the authors proved a complete characterization of the 1-dimensional \emph{Dowker persistence diagrams} of cycle networks. From our experiments on cycle nodes having between 3 and 20 nodes, it appears that the 1-dimensional PPDs of cycle networks admit the same characterization result. This is expressed in Conjecture \ref{conj:pph-cycle}.

\subsection{U.S. economic sector data} \label{sec:us-econ} Each year, the U.S. Bureau of Economic Analysis (\url{https://www.bea.gov/}) releases two datasets---the ``make" and ``use" tables---that list the production of commodities by industries, as well as the flow and usage of these commodities across the supply chain \cite{econ}. Economic agencies then analyze this data to obtain a yearly review of the U.S. economy. We obtained use table data for 15 industries across the year range 1997-2015 from \url{https://www.bea.gov/industry/io_annual.htm} (last accessed December 15, 2016). Because this dataset shows the yearly (asymmetric) flow of commodities across industries, it forms a collection of 19 directed networks (one for each year between 1997 and 2015). As a proof-of-concept for our PPH method applied to a real dataset, we asked the following question: Can the PPDs obtained from this dataset capture a major economic event, such as the 2007-2008 financial crisis? 

The use table dataset consists of a list of 15 industries defined according to the 2007 North American Industry Classification System (NAICS), and the flow of 15 categories of commodities across these industries. In accordance with NAICS, the same labels are used for both industries and commodities, e.g. ``Mining" is a label for both the mining industry and the mining commodities. The use table is a $15\times 15$ matrix where the rows represent commodities, and the columns represent the industries which use the commodities. The sum of all the entries in a row is the total output of the corresponding commodity, and the sum of all the entries in a column is the total amount of product used by the corresponding industry. 

We preprocessed each use table via the following steps: (1) we removed the diagonal to focus on the use of \emph{external} commodities by each industry, (2) we normalized each column to sum to 1, thus obtaining the \emph{fraction} of total consumption that each commodity contributed to a single industry, and (3) we passed each nondiagonal element through the dissimilarity function $f(x) = 1-x$ to obtain a dissimilarity network. Thus we obtained a set of 19 dissimilarity networks, each corresponding to a year in the range 1997-2015. 

We then computed the 1-dimensional PPD of each of these 19 networks, and then computed a $19\times 19$ matrix of pairwise bottleneck distances. Then we applied single linkage hierarchical clustering to this matrix and obtained the dendrogram in Figure \ref{fig:econ}. Notice that the 1-dimensional PPD for 2008 is significantly distinct from those of all the other year ranges. This indicates precisely what one would expect, given the financial turmoil of 2007-2008 that triggered a global recession. Furthermore, it appears that all the years before and including 2006 form a coherent cluster when taking a vertical slice at resolution $0.01$. This cluster also includes 2015, which may indicate that the U.S. economy has finally stabilized (while contentious, some economic experts may consider this to be ground truth \cite{econ2015}). So it appears that PPDs can capture meaningful information from real-world asymmetric data, thus validating their use as a tool for analyzing directed networks.

\section{Discussion}

To our knowledge, persistent path homology is the most asymmetry-sensitive method of computing persistent homology of directed networks in the existing literature. At the same time, its construction is sufficiently distinct from existing simplicial constructions of persistent homology to merit an independent discussion of its characteristics. We consider this paper to be a first announcement of results in a research program devoted to studying both the computational and theoretical aspects of the PPH method. In a forthcoming update to this paper, we plan to provide a deeper theoretical study of this method. A software package for public use and demonstrations on additional real-world directed networks will be made available through \url{https://research.math.osu.edu/networks/}.

\medskip
\paragraph{\textbf{Acknowledgments}} This work was supported by NSF grants IIS-1422400 and CCF-1526513. Facundo M\'emoli was supported by the Mathematical Biosciences Institute at The Ohio State University. We thank Pascal Wild for correcting an error on an earlier version of this paper.

\bibliographystyle{alpha}
\bibliography{biblio}

\end{document}